\documentclass[11pt]{article}
\usepackage{amssymb, amsmath,amsthm,cases, color,graphicx,enumitem,float}
\usepackage{epstopdf}
\usepackage{color}
\usepackage{bm}
\author{Valentin S. Afraimovich\thanks{Universidad Autonoma de San Luis Potosi, IICO}, 
Gregory Moses\thanks{Corresponding author, Department of Mathematics, Ohio University, \ttfamily{gm192206@ohio.edu}}, 
Todd Young\thanks{Ohio University, Department of Mathematics}}
\title{Two dimensional heteroclinic attractor in the generalized Lotka-Volterra system.}
\usepackage[section]{placeins}
\newtheorem{defn}{Definition}[section]

\newtheorem*{thm*}{Theorem}
\newtheorem{thm}[defn]{Theorem}

\newtheorem{cor}[defn]{Corollary}
\newtheorem{lem}[defn]{Lemma}
\newtheorem{prop}[defn]{Proposition}

\newtheorem{rem}[defn]{Remark}

\begin{document}

\maketitle

\begin{abstract}
\noindent We study a simple dynamical model exhibiting sequential dynamics. We show that in this model
there exist sets of parameter values for which a cyclic chain of 
saddle equilibria, $O_k$, $k=1, \ldots, p$,  have two dimensional unstable manifolds that contain orbits 
connecting each $O_k$ to the next two equilibrium points $O_{k+1}$ and $O_{k+2}$ in the chain 
($O_{p+1} = O_1$). We show that  the union of these equilibria
and their unstable manifolds form a $2$-dimensional surface with boundary that is 
homeomorphic to a cylinder if $p$ is even and a M\"{o}bius strip if $p$ is odd.
If, further, each equilibrium in the chain satisfies a condition called ``dissipativity,"
then this surface is asymptotically stable.
\end{abstract}

\section{Background}

In the last decade it became clear that typical processes in many neural and cognitive networks are 
realized in the form of sequential dynamics 
(see \cite{RHL}, \cite{JFSMK}, \cite{RVTA}, \cite{RVLHAL}, \cite{RVSA}, \cite{ATHR} and 
references therein).  The dynamics are exhibited as sequential switching among metastable states, each of 
which represents a collection of simultaneously activated nodes in the network, so that at most 
instants of time a single state is activated. Such dynamics are
consistent with the winner-less competition principle \cite{RVLHAL,ATHR}.  
In the phase space of a mathematical model of such a system each 
state corresponds to an invariant saddle set and the switchings are determined by trajectories joining these 
invariant sets.  In the simplest case these invariant sets may be saddle equilibrium points coupled by heteroclinic 
trajectories, and they form a heteroclinic sequence (HS).  This sequence can be stable if all saddle equilibria have 
one-dimensional unstable manifolds \cite{ATHR}, in the sense that there is an open set of initial points such that 
trajectories going through them follow the heteroclinic ones in the HS, or unstable if some of the unstable 
manifolds are two-or-more dimensional.  In the latter case properties of trajectories in a neighborhood of the HS 
were studied in \cite{ATHR} and \cite{AVTR}.  General results on the stability of heteroclinic sequences have been obtained by Krupa and Melbourne \cite{melbourne1,melbourne2}
in the form of necessary conditions that may also be sufficient in the presence of certain algebraic conditions.

An instability of a HS may be caused if some initial conditions follow trajectories on the unstable 
manifold different from the heteroclinic ones.  M.~Rabinovich suggested \cite{R} considering the case when all 
trajetories on the unstable manifold of any saddle in a HS are heteroclinic to saddles in the HS, which implies an 
assumption that the HS is, in fact, a heteroclinic cycle.  In this case one can expect some kind of stability, not 
of the HS, of course, but of the object in the phase space formed by all heteroclinic trajectories of the 
saddles in the HS.  We study in our paper the Rabinovich problem.  We deal here with the generalized 
Lotka-Volterra model \cite{ATHR} that is a basic model of sequential dynamics for which unstable sets are 
realized as saddle equilibrium points. All variables and parameters may take only nonnegative values, 
so we work in the positive orthant of the phase space $\mathbb{R}^n$.  We impose some restrictions on 
parameters under which all unstable manifolds of the saddle point are two-dimensional and all trajectories on 
them (in the positive orthant) are heteroclinic in some specific way (see below).  We prove that they form a 
piece-wise smooth manifold homeomorphic to the cylinder if the number of the saddle points in the HS is 
even or to the M\"{o}bius band if it is odd.  We prove also that under the additional assumption that 
each equilibrium is {\em dissipative} (see below), then this manifold is the maximal attractor for some 
absorbing region (in the positive orthant).  Trajectories in this region may follow different heteroclinic 
trajectories and may manifest some kind of weak complex behavior.

Although our motivations are neurological, we observe that heteroclinic networks (and thus, potentially, high-dimensional counterparts of the same) are ubiquitous, appearing 
in applications that range from celestial dynamics \cite{koon} to evolutionary game theory \cite{hofbauer}.


\section{Notation and Results}

We begin our consideration of two-dimensional heteroclinic channels with the study of a series of Lotka-Volterra equations.  Lotka-Volterra models are widely used in the context of heteroclinic sequences where
all the unstable manifolds are one dimensional, so that each equilibrium is connected to exactly one subsequent equilibrium (e.g. \cite{horchler,schwappach,gonzalez,guckenheimer}).  It has been recently shown that they ask provide a general method for embedding directed graphs into a system of ordinary differential equations \cite{Ashwin1}.  We remark that although \cite{Ashwin1} allows graphs of high valency to be modeled by heteroclinic networks, and some work has been done on the stability of such systems (e.g. \cite{kirk}, which considers the competing dynamics of the ``overlapping" channels $\lambda_1 \rightarrow \lambda_2 \rightarrow \lambda_3 \rightarrow \lambda_1$ and $\lambda_1 \rightarrow \lambda_2 \rightarrow \lambda_4 \rightarrow \lambda_1$), study of heteroclinic networks has usually only considered one-dimensional unstable manifolds.  In keeping with the discussion of the introduction, we consider the dynamics of a system such that initial conditions arbitrarily close to any of $p \le n$ saddle nodes may be mapped into neighborhoods of either one of two other saddle nodes.  In particular, a simple general model for heteroclinic sequential dynamics was given in \cite{Afro} by
\begin{equation}
\label{eq:vectorfield}
\dot{x}_i = F_i(x) = x_i(\sigma_i - \sum_{j=1}^{n}\rho_{ij}x_j) \quad \text{ for } i=1,...,n,
\end{equation}
where all of the parameters are assumed to be positive.  
The constants $\rho_{ij}$ have biological meaning, representing inhibition of mode $i$ by mode $j$.  
For the sake of simplicity, we assume that $\rho_{ii}=1$ for all $i$.  Further, since 
the variable $x$ is assumed to encode biological information that is necessarily 
non-negative, e.g.\ activation levels or chemical concentrations, we restrict the system 
to the first closed orthant,  
$\overline{\mathbb{R}_+^n} =\{x \in \mathbb{R}^n: x_i \ge 0, 1 \le i \le n\}$.
The system is so constructed that it contains saddle points lying on the axes, with $\sigma_i$ 
being the $i$-th coordinate of the $i$-th saddle along the $i$-th axis ($\sigma_i>0$),  
i.e., the system \eqref{eq:vectorfield} has $n$ equilibrium points of the form 
$O_k=(0, ..., 0, \sigma_k, 0,  ... )$, for $k = 1,..,n$.  There may be other equilibria 
as well, but they are not relevant for our purposes; we only study transitions between the $n$ equilibria just defined.

In the present work we suppose that the first $p$ equilibria points are sequentially 
connected by a set of $2$-dimensional unstable manifolds.  For each $k$, $1 \le k \le p$,
there will be a heteroclinic orbit connecting $O_k$ to $O_{k+1}$ and a heteroclinic orbit
connecting $O_k$ to $O_{k+2}$.  Furthermore, the  system is closed in the sense 
that $O_{p+i}=O_i$, i.e. $p$ is the modulus of the subscript.
In the following,  we will consider the restrictions necessary to enforce such dynamics. 

\begin{figure}
\centering
\includegraphics[width=\textwidth]{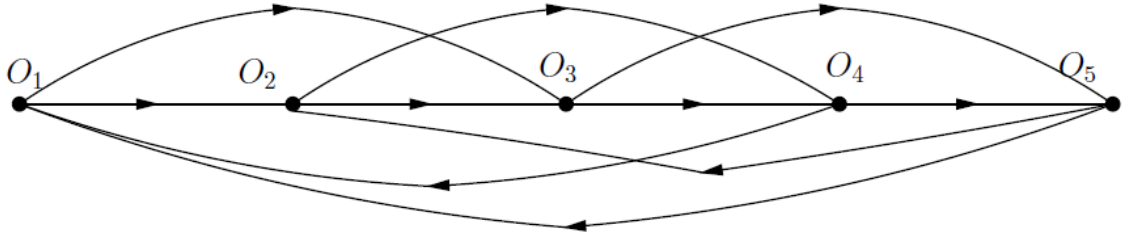}
\caption{A rough representation of the dynamics of the system defined by Equation~\ref{eq:vectorfield}, with $p=5$.}
\end{figure}

To ensure that there are heteroclinic trajectories between $O_k$ and both 
$O_{k+1}$ and $O_{k+2}$, we apply eigenvalue conditions to the system.  
Consider first $O_1$. The linearization of the vector field $F(x)$ \eqref{eq:vectorfield} at $O_1$ is given by the upper triangular matrix:
\begin{equation}\label{DFO1}
DF(O_1) = \left( \begin{array}{cccccc}
-\sigma_1  & - \sigma_1 \rho_{12} & - \sigma_1 \rho_{13} & - \sigma_1 \rho_{14} &  \cdots         & - \sigma_1 \rho_{1n}  \\
     0           &  \sigma_2-\rho_{21}\sigma_1 &        0            &               0                  &    \cdots       &        0                       \\
     0           &           0          &  \sigma_3-\rho_{31}\sigma_1 &              0                 &   \cdots          &        0                       \\
     0           &           0          &                                0          &  \sigma_4-\rho_{41}\sigma_1 &  \cdots  &      \vdots               \\
   \vdots     &       \vdots      &                               \vdots  &                0                     &     \ddots      &         \vdots             \\
     0           &           0          &                                0          &                0                    &       \cdots     &          \sigma_n-\rho_{n1}\sigma_1 
     \end{array} \right),
\end{equation}
and so the eigenvalues appear on the diagonal.
The matrices $DF(O_k)$ have a similar simple structure, zeros everywhere except
on the diagonal and on the $k$-th row.
It is then easy to see that the eigenvalues of $DF$
at $O_k$ are 
$$ 
\lambda_k^k=-\sigma_k \quad \text{ and } \quad\lambda_j^k=\sigma_j - \rho_{jk}\sigma_k, \, \, \text{ for } j \ne k.
$$  
In particular, the eigenvalues are all real. One can also see from the structure of \eqref{DFO1} that $DF(O_k)$ 
has a full set of eigenvectors, even if some eigenvalues are repeated.

Note that because of the particular form of the equations, all coordinate
axes, planes and hyperplanes are invariant. 
Thus, in order that trajectories can travel from $O_k$ to $O_{k+1}$ or $O_{k+2}$, 
it is sufficient to put the restriction $0 < \lambda_{k+2}^k, \lambda_{k+1}^k$, 
and $\lambda_j^k<0$ otherwise, to ensure that they can only go in those directions.  
In particular, for each $k$, $1 \le k \le p$, we require that
\begin{gather}
0  <  \min_{i=1,2} \{ \sigma_{k+i}-\rho_{k+i,k}\sigma_k \} \, (\text{indices mod } p), \label{ew_unstable}\\
\lambda_j =  \sigma_j-\rho_{jk}\sigma_k<0, \quad \text{ for } j \ne k, k+1, k+2 \mod p.  \label{ew_stable}
\end{gather}
Note also  that $-\sigma_k <0$.
These inequalities guarantee that each equilibrium is a hyperbolic saddle with $2$ unstable directions and
$n-2$ stable directions.
Let $\widetilde{W}_k^u=W^u(O_k) \cap \overline{\mathbb{R}_+^n}$ be the unstable 
manifold of $O_k$ restricted to the positive orthant.    We show below that 
$$
\Gamma \equiv \bigcup_{k=1}^p (\widetilde{W}_k^u \cup O_k) 
$$
forms a piecewise smooth surface that we will classify topologically as follows. 
\begin{thm}
\label{formofgamma}
Suppose that inequalities \eqref{ew_unstable} and \eqref{ew_stable} hold for each $k$, $1 \le k \le p$ and that each unstable manifold $\widetilde{W}^u(O_k)$ is contained 
in a compact forward invariant set as specified in Lemma~\ref{positivelyinvariantset} (see also Remark 3.11).
When $p$ is even, the union of unstable manifolds $\Gamma$ is homeomorphic to a cylinder.  
When $p$ is odd, $\Gamma$ is homeomorphic to a M\"{o}bius strip
\end{thm}
The proof of this theorem, which is slightly involved, is left for the appendix.

Consider the following definition.
\begin{defn}
Let $\Sigma^u$ and $\Sigma^s$ be the set of stable and unstable eigenvalues, respectively, of
the linearization of a vector field at a saddle equilibrium, i.e., $\max Re(\Sigma^s) < 0$ and $\min Re(\Sigma^u) >0$.
We say that the saddle is {\bf dissipative} if 
$$
   \max  Re(\Sigma^u)  < - \max Re(\Sigma^s) .  
$$
In other words, the weakest stable eigenvalue is stronger than the strongest unstable eigenvalue.
 (See \cite{Afro}.)
\end{defn}
In terms of the specific vector field under study all the eigenvalues in question are real 
and for each $k$ we have:
\begin{equation}\label{dis}
  \max_{i=1,2} \{ \sigma_{k+i}-\rho_{k+i,k}\sigma_k  \} 
        < \min_{j \neq k, k+1, k+2}  \{ |\sigma_j - \rho_{jk}\sigma_k|, \sigma_k \} \quad \text{ (indices mod $p$)}.
\end{equation}

The main goal of this manuscript is to show that, under the condition that each saddle 
equilibrium is dissipative, then $\Gamma$ is asymptotically stable. Specifically, our main theorem is:
\begin{thm}
\label{maintheorem}
Suppose that inequalities \eqref{ew_unstable}, \eqref{ew_stable} and \eqref{dis} hold for each $k$, $1 \le k \le p$ and that each unstable manifold $\widetilde{W}^u(O_k)$ is contained 
in a compact forward invariant set as specified in Section~\ref{section:positive_invarian}.
Then $\Gamma$ is asymptotically stable.
\end{thm}

The proof of Theorem~\ref{maintheorem} breaks roughly into two independent pieces.  We start by considering the trajectory of a representative point that is $\epsilon$-close to $\Gamma$, but is distant from each of the fixed points $O_i$.  In such a case, the dynamics of the system are controlled largely by three consecutive saddles $O_i$, $O_{i+1}$, and $O_{i+2}$.  In Section~\ref{section:3D}, we consider the restriction of  \eqref{eq:vectorfield} to three consecutive dimensions; we will gain information on the full-dimensional system by viewing it as a perturbation of this restriction.  The second part of the proof is to consider the dynamics as a trajectory passes near a fixed point; we consider this in Section~\ref{section:local}.

In Section~\ref{stablesurface} we prove the main theorem. In Section~\ref{parsers} we show that there is a
non-empty parameter set for which the conditions of the theorem are satisfied.


\section{Three Dimensions}
\label{section:3D}

We begin our proof of Theorems 2.1 and 2.3 with a study of the restriction of the system 
to three dimensional sub-spaces corresponding to three consecutive coordinate directions.  
Through a series of geometric lemmas, we prove the main result of the section, 
Theorem~\ref{3D}, which provides information on the behavior of trajectories inside 
this invariant subspace.  This will be used in later sections to complete the proofs of the 
main results by providing information on those parts of the 
full phase space where all but three coordinates are small.

\subsection{Set-up}

Let $O_i$, $O_{i+1}$, and $O_{i+2}$ be any three consecutive equilibria and restrict the 
system (\ref{eq:vectorfield}) to the three dimensions spanned by consecutive coordinates
$x_i$, $x_{i+1}$, and $x_{i+2}$.  For 
convenience, we will refer to these three variables as $x_1$, $x_2$, and $x_3$.  
Our goal in three dimensions is to show the existence of a compact, forward invariant 
set containing $O_1$, $O_2$, and $O_3$ such that any trajectory with an initial value 
in the interior of that set converges to $O_3$.  

Restricted to three dimensions, the  equations \eqref{eq:vectorfield} are reduced to
\begin{equation}
\label{threedimensionalrestriction}
\begin{split}
\dot{x}_1 = x_1(\sigma_1 - x_1 -\rho_{12}x_2 - \rho_{13}x_3),\\
\dot{x}_2 = x_2(\sigma_2 - x_2 -\rho_{21}x_1 - \rho_{23}x_3),\\
\dot{x}_3 = x_3(\sigma_3 - x_3 -\rho_{31}x_1 - \rho_{32}x_2).
\end{split}
\end{equation}
The restriction of the dimension of the unstable manifolds via the eigenvalue conditions, 
and the positivity conditions on $\sigma_1, \sigma_2, \text{and } \sigma_3$ yield the following 
inequalities:
\begin{gather}
-\sigma_1 < 0 < \sigma_j - \rho_{j1}\sigma_1, \, \, \, j = 2,3,  \label{ev1}\\
-\sigma_2 < 0 <\sigma_3-\rho_{32}\sigma_2, \label{ev2}\\
\sigma_1-\rho_{12}\sigma_2<0, \label{ev3}\\
\sigma_1-\rho_{13}\sigma_3<0, \label{ev4}\\
\sigma_2-\rho_{23}\sigma_3<0. \label{ev5}
\end{gather}
Of those inequalities, (\ref{ev1}) controls the behavior of the system at $O_1$, 
(\ref{ev2}) - (\ref{ev3}) control the behavior of the system at $O_2$, and 
(\ref{ev4}) - (\ref{ev5}) at $O_3$.  The point $(\sigma_1,0,0)$ is a saddle with a 
two-dimensional unstable manifold, $(0,\sigma_2,0)$ is a saddle with a 
one-dimensional unstable manifold, and $(0,0,\sigma_3)$ is a sink for the system~(\ref{threedimensionalrestriction}).  We remark that although \eqref{threedimensionalrestriction} has the form of the May-Leonard model, the particular parameter restrictions under consideration yield simple dynamics (see Theorem~\ref{3D}), and prevent the more complex behavior usually studied in that context.  In particular, they are inconsistent with the symmetric May-Leonard model as it was introduced in \cite{may_leonard}.

\begin{figure}[!htb]
\centering
\includegraphics[width=70mm]{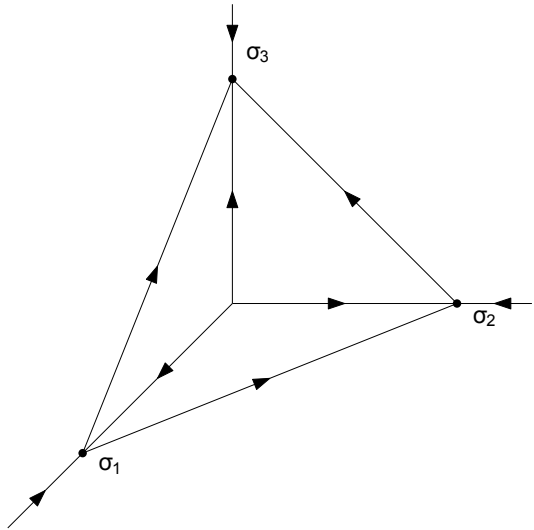}
\caption{Illustration of dynamics of the system projected onto the first three 
coordinates.  All consecutive triplets $(x_i, x_{i+1}, x_{i+2})$ (where $x_{p+i}=x_i$) 
of coordinates possess the same qualitative dynamics.}
\end{figure}

For each $i$, we will be interested in the points where $\dot{x}_i=0$; each such 
set is the union of the plane $x_i=0$ and some ``nontrivial" plane.  We designate those planes:
\begin{align}
P_1 := \{ \sigma _1 - \rho_{12}x_2 - \rho _{13}x_3 - x_1=0 \},\\
P_2 : = \{ \sigma _2 - \rho_{21}x_1 - \rho _{23}x_3 - x_2=0 \},\\
P_3 : = \{ \sigma _3 - \rho_{31}x_1 - \rho _{32}x_2 - x_3=0 \},
\end{align}
where $\dot{x}_i=0$ on $P_i$.

We will also refer to the plane passing through the points $(\sigma_1 , 0,0)$, 
$(0,\sigma _2 , 0)$, and $(0,0,\sigma _3)$, which we denote by $\Sigma$.  
Observe that $\Sigma$ is given by the equation
\begin{align*}
\Sigma:=\frac{x_1}{\sigma _1} + \frac{x_2}{\sigma _2} + \frac{x_3}{\sigma _3} = 1.
\end{align*}

For ease of discussion, we will also name the coordinate planes:
\begin{align*}
&P_{12}:=\{(x_1, x_2, x_3): x_3=0\},\\
&P_{23}:=\{(x_1, x_2, x_3): x_1=0\}, \\
&P_{13}:=\{(x_1, x_2, x_3): x_2=0\}.
\end{align*}

Because of the positivity conditions on the parameters and variables, we may 
use $P_1$, $P_2$, $P_3$, and $\Sigma$  as shorthand for the intersection of those 
planes with the first octant without the risk of confusion.  We observe that the 
intersection of $\Sigma$ and of each $P_i$ with $\overline{\mathbb{R}_+^n}$ 
is a compact triangle, a fact we will use repeatedly in the following section.

In order to describe the dynamics of orbits, we are interested in when one plane 
lies ``above" another in the first octant.  Consider the following definition:
\\

\begin{defn}
Observe that each plane $P_i$ can be written as the graph of a function $z^i(x_1,x_2)$.  
A plane $P_i$ {\bf dominates} a plane $P_j$ if $(x_1, x_2, z^i(x_1,x_2)) \in P_i$ and 
$(x_1,x_2,z^j(x_1,x_2)) \in P_j$ implies that $z^j(x_1, x_2)<z^i(x_1, x_2)$ for all $x_1, x_2>0$.  
A plane $P_i$ {\bf is dominated by} a plane $P_j$ if $P_j$ dominates $P_i$.
\end{defn}


\subsection{Geometrical Lemmas}

Each plane $P_i$ divides $\overline{\mathbb{R}_+^3}$ into two regions, one where $\dot{x}_i$ 
is positive and another where it is negative.  This allows information about 
$\dot{x}_i$ to be gained from purely geometric information.  For example, $\dot{x}_1$ is 
positive below $P_1$, and negative above it.  Since $P_3$ 
dominates $P_1$ (i.e. is always above it), we instantly see that $\dot{x}_1|_{P_3}<0$ (see Corollary~\ref{x1negativeonP3}, given below).  
We introduce geometric lemmas giving the information we can gain in this way; the proofs of 
all of them are parallel to one another, and can be summarized as follows: since the system is 
restricted to the first octant, we compare two planes (triangles) by seeing where they intersect the 
$x_1$, $x_2$, and $x_3$ axes.  Denote the compact triangle thus 
formed by a plane $S$ as $T_S$ and the non-zero component of its vertex on the $i$-th axis as $S^i$.  Then a plane 
$S$ dominates a plane $R$ if $R^i \le S^i$ for $i=1,2,3$, with at least one of those a strict equality, i.e. if its vertices are farther from the origin.

\begin{lem}
The plane $P_1$ is dominated by the plane $\Sigma$.
\end{lem}
\begin{proof} We consider where each plane intersects each axis:
\begin{itemize}\itemsep1pt

\item The planes $P_1$ and $\Sigma$ both intersect the $x_1-$axis at the point $O_1$.

\item The plane $P_1$ intersects the $x_2$ axis at $(0,\frac{\sigma_1}{\rho_{12}},0)$, 
while $\Sigma$ intersects the axis at $O_2$.  We know from (\ref{ev3}) that 
$\frac{\sigma_1}{\rho_{12}}<\sigma_2$.

\item The plane $P_1$ intersects the $x_3-$axis at $(0,0,\frac{\sigma_1}{\rho_{13}})$, 
while $\Sigma$ intersects the axis at $O_3$.  We know from (\ref{ev4}), that 
$\frac{\sigma_1}{\rho_{13}}<\sigma_3$.
\end{itemize}
Since $P_1^i \le \Sigma^i$ for all $i$, $P_1$ is dominated by $\Sigma$.
\end{proof}

\begin{lem}
\label{P3dominatesSigma}
The plane $\Sigma$ is dominated by $P_3$.
\end{lem}
\begin{proof}  We consider where each plane intersects each axis:
\begin{itemize} \itemsep1pt

\item The plane $P_3$ intersects the $x_1-$axis at the point $(\frac{\sigma_3}{\rho_{31}},0,0)$, 
while $\Sigma$ intersects the axis at $O_1$.  We know from (\ref{ev1}) that 
$\sigma_1<\displaystyle \frac{\sigma_3}{\rho_{31}}$.

\item The plane $P_3$ intersects the $x_2$ axis at $(0,\frac{\sigma_3}{\rho_{32}},0)$, 
while $\Sigma$ intersects the axis at $O_2$.  We know from (\ref{ev2}) that 
$\sigma_2<\frac{\sigma_3}{\rho_{32}}$.

\item The planes $P_3$ and $\Sigma$ both intersect the $x_3$ axis at $O_3$.
\end{itemize}
Since $\Sigma^i \le P_3^i$ for all $i$, $P_3$ dominates $\Sigma$.
\end{proof}

The property ``is dominated by" is clearly transitive, so the following corollary holds.

\begin{cor}
\label{P3dominatesP1}
$P_1$ is dominated by $P_3$.
\end{cor}
Since $\dot{x}_1<0$ above the $P_1$ plane, the following corollary follows immediately.

\begin{cor}
\label{x1negativeonP3}
 On the plane $P_3$, $\dot{x}_1<0$.
\end{cor}

If we further had that $\Sigma$ dominates $P_2$, then the eigenvalue conditions 
introduced in \cite{Afro} and summarized as (3) - (8) would be sufficient to ensure 
the existence of a positively invariant region.  It happens, however, that  
this is not the case.

\begin{lem}
\label{SigmaP2}
The plane $\Sigma$ neither dominates nor is dominated by $P_2$.
\end{lem}
\begin{proof} We consider where each plane intersects each axis:
\begin{itemize}\itemsep1pt

\item $P_2$ intersects the $x_1-$axis at the point $(\frac{\sigma_2}{\rho_{21}},0,0)$, while $\Sigma$ intersects the axis at $O_1$.  We know from (\ref{ev1}) that $\sigma_1<\frac{\sigma_2}{\rho_{21}}$, and therefore $\Sigma$ does not dominate $P_2$.

\item $P_2$ intersects the $x_3-$axis at $(0,0,\frac{\sigma_2}{\rho_{23}})$, while $\Sigma$ intersects the axis at $O_3$.  We know from (\ref{ev5}) that $\frac{\sigma_2}{\rho_{23}}<\sigma_3$, and therefore $P_2$ does not dominate $\Sigma$.
\end{itemize}
Thus, neither plane dominates the other.
\end{proof}

The situation is somewhat salvaged by the following.
\begin{lem}
\label{ParameterRestriction}
If $\frac{\sigma_2}{\rho_{21}} \le \frac{\sigma_3}{\rho_{31}}$, then $P_2$ is dominated by $P_3$.  
\end{lem}
\begin{proof}  
In the proof of Lemma~\ref{P3dominatesSigma} (second bullet point), we established that $P_2^2 \le P_3^2$.  In the proof of Lemma~\ref{SigmaP2} (second bullet point), we established that $P_2^3 \le P_3 ^3$.  All that remains for $P_3$ to dominate $P_2$ is for $P_2^1 \le P_3^1$, which occurs if and only if  $\frac{\sigma_2}{\rho_{21}} \le \frac{\sigma_3}{\rho_{31}}$.
\end{proof}
The parameter restriction in \eqref{ParameterRestriction} written in terms of the general systems gives that
for each $k$, $1 \le k \le p$, 
\begin{equation}\label{P23}
\frac{\sigma_{k+1}}{\rho_{k+1,k}} \le \frac{\sigma_{k+2}}{\rho_{k+2,k}}  \quad \text{ (indices mod $p$)}.
\end{equation}

Similarly to Corollary~\ref{x1negativeonP3}, we have the following:

\begin{cor}
\label{x2negativeonP3}
In the region of parameter space where the hypotheses 
of Lemma~\ref{ParameterRestriction} are satisfied, $\dot{x}_2 | _{P_3}<0$.
\end{cor}

We note here one additional observation.
\begin{lem}\label{box}
The rectangular box:
$$
   B = \{ x: 0 \le x_i \le \sigma_i, i=1, \ldots, n \}
$$
is forward invariant with respect to the system \eqref{eq:vectorfield}. Further,
the unstable manifolds $\widetilde{W}_k^u$ are all contained in $B$.
\end{lem}
Forward invariance follows immediately from the differential equations \eqref{eq:vectorfield} and the assumption that $\rho_{ii} = 1$.
The conclusion that the unstable manifold at $O_1$ is inside $B$ follows easily by noting that the unstable eigenspace
at $O_1$ (restricted to the first orthant) is strictly inside $B$.

\begin{figure}[!htb]\label{fig:domination}
\centering
\includegraphics[width=90mm]{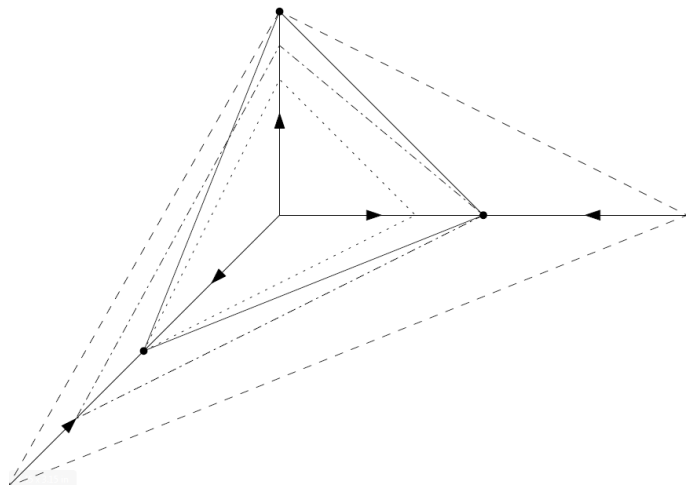}
\caption{The plane $P_1$ (dotted) is dominated by $\Sigma$ (solid) which is dominated by 
$P_3$ (dashes).  If the hypotheses of  Lemma~\ref{ParameterRestriction} are satisfied, 
then the plane $P_2$ (dashes and dots) is dominated by $P_3$.}
\end{figure}

Rather than requiring that $P_3$ dominates $P_2$, we could require only that $P_3$ dominates 
$P_2$ inside the forward invariant box $B$. 
Since $P_3^1$ is strictly outside of $B$ by Lemma~\ref{P3dominatesSigma} 
(see Figure~\ref{fig:domination}), this produces a strictly greater set of allowable
parameter values.


\subsection{Existence of a Positively Invariant Region}
\label{section:positive_invarian}
We remarked that our goal in three dimensions was to show the existence of a compact, forward-invariant set whose trajectories converge to $O_3$.  We now carry this out.

\begin{lem}
\label{positivelyinvariantset}
In the region of parameter space where the inequalities (\ref{ev1}) - (\ref{ev5}) and the 
hypotheses of Lemma~\ref{ParameterRestriction} are satisfied, the planes $P_3$, $P_{23}$, 
$P_{13}$, and $P_{12}$ enclose a positively invariant region, i.e. no trajectory leaves in positive time.
\end{lem}
\begin{proof}
No trajectory can leave through any of the $x_i=0$ planes, 
since $\dot{x}_i | _{x_i \text{-plane}}=0$.  
The outward normal vector to $P_3$ is $N=\langle \rho_{31} , \rho_{32} , 1 \rangle$, 
and its scalar product with the vector field \eqref{eq:vectorfield} is given by 
$\bar{n} \cdot F=\rho_{31}\dot{x}_1+\rho_{32}\dot{x}_2$ when $\dot{x}_3=0$.  
This is negative by Corollary~\ref{x1negativeonP3} and Corollary~\ref{x2negativeonP3}.
\end{proof}

\begin{rem}
\label{rem:forward}
Since the pair of inequalities $\rho_{31}\dot{x}_1<0$ and $\rho_{32}\dot{x}_2<0$ form a sufficient, but not 
necessary, condition for the inequality $\rho_{31}\dot{x}_1+\rho_{32}\dot{x}_2<0$ to be satisfied, 
a positively invariant region might exist even if the hypothesis of Lemma~\ref{ParameterRestriction} is not satisfied.  
For instance, since the box $B$ (Lemma~\ref{box}) is forward invariant in may be that
$\dot{x}_2|_{P_3 \cap B}$ is negative even when the hypotheses of Lemma~\ref{ParameterRestriction} fail.
Lemma~\ref{positivelyinvariantset} can therefore be extended to a region of parameter space 
that includes the region defined by Lemma~\ref{ParameterRestriction} as a proper subset.
\end{rem}
\begin{thm}
\label{3D}
Suppose that $P_3$, $x_1=0$, $x_2=0$, and $x_3=0$ enclose a positively invariant region.  
Any trajectory $\phi^t(x)$ in the above-described region that is not contained in $P_{12}$ (the $x_3=0$ plane) goes 
to $(0,0,\sigma_3)$ as $t \rightarrow +\infty$.
\end{thm}

\begin{proof}
Any trajectory in a compact positively invariant region has a non-empty $\omega$-limit set.  Fix a trajectory with initial condition $x$ in the interior of the region, and let $q$ be an arbitrary point in the $\omega$-limit set $\Omega_x$.  The proof breaks into three parts: 
we prove that $q$ lies on $P_3$, that it lies on $P_{23}$, and that it lies on $P_{13}$; 
the intersection of these planes is $(0,0,\sigma_3)$.  The proofs of the second and third 
statements are essentially the same as the proof of the first statement, which we treat in detail.

By way of contradiction, suppose that $q$ does not lie on $P_3$.  First of all, note that $q$ does 
not lie on $P_{12}$, because by assumption, the $x \notin P_{12}$, and since the trajectory increases in the $x_3$ variable, $\phi^t(x)$ cannot approach $P_{12} = \{x_3 = 0 \}$

Since $\dot{x}_3$ is continuous, and $q$ lies on 
neither $P_3$ nor $P_{12}$, the regions where $\dot{x}_3=0$, we can find a spherical 
neighborhood $O_r(q)$ centered at $q$ with radius $r$ such that 
$\dot{x}_3 |_{O_r(q)} >\epsilon >0$ for some $\epsilon$.  It follows that $q$ is not a fixed point, and $\Omega_x$ is not the singleton $\{q\}$, since $\omega-$limit sets are forward invariant \cite{perko}.

The $x_3$ component of $\phi^t(x)$ is non-decreasing
in the positively invariant set and it is bounded above, since it is bounded above by $\sigma_3$.  Thus it has a limit $x^*_3$.
It follows by continuity that the $x_3$ component of any point in $\Omega_x$
is also $x^*_3$.  We have supposed that $q \in \Omega_x$ but $q \notin P_3$;
now consider $\phi^t(q)$. Since $\dot{x}_3 >0$ at $q$, and thus
in a neighborhood of $q$, the $x_3$ component of $\phi^t(q)$ 
must be strictly increasing as a function of time along the forward 
solution near $q$. This contradicts
that the $x_3$ component is $x^*_3$ everwhere on $\Omega_x$.

We repeat the argument twice.  First, $q$ must lie on the plane $P_{23}$; otherwise $\dot{x_1}(q)<0$,
and the same contradiction will be obtained.  Then, using the same argument, we see 
that it must lie on $P_{13}$.  We observe here, with reference to 
Lemma~\ref{ParameterRestriction}, that even if $P_2$ is not dominated by $P_3$, 
it is dominated by it on the restriction to $P_{23}$, and the argument 
therefore goes through.  Thus $q$ must 
lie on the single point, $(0,0,\sigma_3)$, and our proof is complete.
\end{proof}


\section{The Dynamics in the Full Phase Space}
\label{section:local}

\subsection{Local Dynamics Near an Equilibrium $O_k$}

We assumed in (\ref{dis}) that the saddle points $O_k$ are dissipative. 
Denote by $\nu$ the ratio:
\begin{equation}
    \nu \equiv \frac{ \min |\{Re(\Sigma^s)\}|   }{  \max\{Re(\Sigma^u)\} }.
\end{equation}
We call $\nu$ the {\em minimal saddle value} of the equilibrium (see \cite{Shil}).
The equilibrium is dissipative if $\nu >1$.

 The main implication of this assumption 
is that  a trajectory starting at an initial value at a distance $\epsilon$ from the stable manifold of 
the saddle comes to a point of distance on the order of $\epsilon^\nu$ from the unstable manifold after going 
through a neighborhood of the saddle.  Formulating this result more strictly, we 
label the variables in a neighborhood of $O_k$ into the  $2$-dimensional unstable subspace 
$\eta = (x_{k+1},x_{k+2})$ and the $(n-2)$-dimensional stable subspace 
$\xi = \overline{(x_j)}, j \neq k+1, k+2$. Let $|  \cdot |$ denote the sup norm in these local coordinates.   By the Stable Manifold Theorem, for each $k$ there
exists $\delta_k > 0$ such that in a $\delta_k$-neighborhood of $O_k$, the unstable manifold
$W^u(O_k)$ is the graph of a (smooth) function $\xi = h_k^u(\eta)$. Let $\delta$ be the minimum of
these $\delta_k$ and consider the $\delta$-neighborhood $V_k$ of each $O_k$, $k=1, \ldots, p$

For fixed $k$, and $0 < \epsilon_k$ sufficiently small, define a pair of sections, 
$$
S_0 = \{ (\xi,\eta): |\xi| = \delta, |\eta| \le \epsilon_k \}  \qquad \text{and} \qquad 
S_1 = \{ (\xi,\eta): |\xi| \le \delta, |\eta| = \delta \}. 
$$
By the classical Shil'nikov variables technique (see \cite{Shil}), there exists $\epsilon_k$ sufficiently small
so that every forward solution starting at $S_0$ will intersect $S_1$ before leaving $V_k$. Let $\epsilon$
be the minimum of the $\epsilon_k$'s needed in the neighborhood of each $O_k$. Let $T$ denote the time at which
such a solution intersects $S_1$.

\begin{thm}
\label{ConvergenceToManifold}
If  $\delta$ and $\epsilon$ are sufficiently small and if $(\xi(0),\eta(0)) \in S_0$ 
and $(\xi(T),\eta(T)) \in S_1$ then $|\xi(T)| \le C|\eta(0)|^{\nu-e}$, where $C>0$ is independent of the initial point and $\nu -e>1$.
\end{thm}

The sections $S_0$ and $S_1$, together with Theorem~\ref{ConvergenceToManifold}, are illustrated in two dimensions in Figure~\ref{fig:schematics}
\begin{figure}
\centering
\includegraphics[width=.5\textwidth]{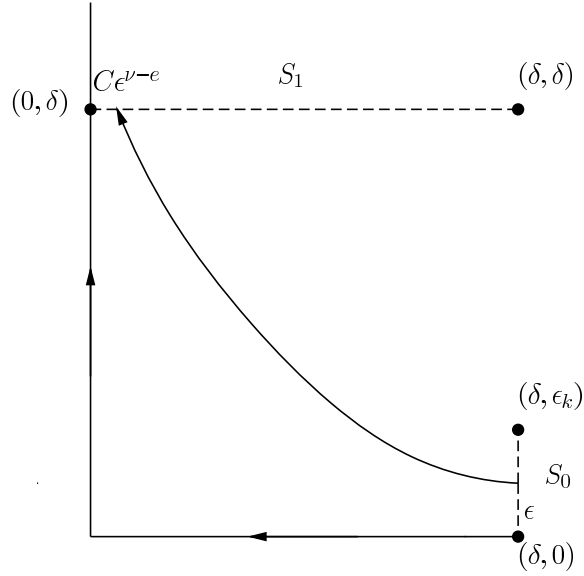}
\label{fig:schematics}
\caption{A schematic diagram of Theorem~\ref{ConvergenceToManifold}.  Initial conditions on a section intersecting a stable manifold are mapped in finite time to a section intersecting an unstable manifold, resulting in a contraction.}

\end{figure}

\begin{proof}
For ease of notation, consider $O_1$. 
Note from \eqref{DFO1} that the eigenvector corresponding to the first diagonal element, $- \sigma_1$,
can have only one non-zero component and that is in the $x_1$ direction. The eigenvector corresponding to the 
$j$-th diagonal element, $j \neq 1$, has non-zero components in the $x_1$ and $x_j$ directions only.

Now take into account that the hyperplane $\{(x_2 = x_3 = 0)\}$ is invariant under the flow of the equations
and tangent to the stable eigenspace of $DF(O_1)$. It thus coincides locally with and contains $W^s(O_1)$.

Next note that the unstable eigenspace has no non-zero components in the coordinate directions $x_4, x_5, \ldots, x_n $
and that the hyperplane where they are zero is invariant under the flow. It thus follows that the $x_4, x_5, \ldots, x_n $
coordinates of the unstable manifold are all zero. Thus the unstable manifold $W^u(O_1)$ is given locally in the
$\delta$-neighborhood of $O_1$ as the graph  $x_1 = \sigma_1 + h^u_1(x_2,x_3)$ where $h(0,0) = 0$ and
$h$ is smooth. Consider the local change of variables that ``straightens out" the unstable manifold:
\begin{equation}\label{coordchange}
   X_1 = x_1 -\sigma_1 - h(x_2,x_3).
\end{equation}
Under this smooth change of coordinates \eqref{eq:vectorfield} becomes:
\begin{equation}
\begin{split}
\dot{X}_1 & = - \sigma_1  X_1 - f(X_1,x_2, \ldots, x_n) \\
\dot{x}_2 & = x_2 \left( \sigma_2 - x_2 - \rho_{21}(X_1+\sigma_1+h(x_2,x_3)) - \rho_{23} x_3 - \rho_{24} x_4 - \cdots - \rho_{2n} x_n \right) \\
\dot{x}_3 & = x_3 \left( \sigma_3 - x_3 - \rho_{31}(X_1+\sigma_1+h(x_2,x_3)) - \rho_{32} x_2 - \rho_{34} x_4- \cdots - \rho_{3n} x_n \right) \\
\dot{x}_4 & = x_4 \left( \sigma_4 - x_4 - \rho_{41}(X_1+\sigma_1+h(x_2,x_3)) - \rho_{42} x_2 - \rho_{43} x_3- \cdots - \rho_{4n} x_n \right) \\
                &  \vdots \\
\dot{x}_n & = x_n \left( \sigma_n - x_n - \rho_{n1}(X_1+\sigma_1+h(x_2,x_3)) - \rho_{n2} x_2 - \rho_{n3} x_3- \cdots - \rho_{n,n-1} x_{n-1} \right).
\end{split}
\end{equation}
Since $W^s(O_1)$ is contained in the coordinate hyperplane, we do not need a change of variables corresponding to $W^s$. 
Noting the $f(0,x_2, \ldots, x_n) \equiv 0$, we use the MVT to define a new function $\frac{f(X_1,x_2,\ldots, x_n) - 0}{X_1 - 0}=f_1(X_1,x_2,\ldots,x_n)$, and rewrite the equations to obtain:
\begin{equation}\label{neweqns}
\begin{split}
\dot{X}_1 & = -\sigma_1 X_1  + f_1(X_1,x_2, \ldots, x_n) X_1 \\
\dot{x}_2 & = x_2 \left( \sigma_2 - \rho_{21}\sigma_1- x_2 - \rho_{21}(X_1+h(x_2,x_3)) - \rho_{23} x_3 - \rho_{24} x_4 - \cdots - \rho_{2n} x_n \right) \\
\dot{x}_3 & = x_3 \left( \sigma_3 - \rho_{31}\sigma_1- x_3 - \rho_{31}(X_1+h(x_2,x_3)) - \rho_{32} x_2 - \rho_{34} x_4- \cdots - \rho_{3n} x_n \right) \\
\dot{x}_4 & = x_4 \left( \sigma_4 - \rho_{41}\sigma_1- x_4 - \rho_{41}(X_1+h(x_2,x_3)) - \rho_{42} x_2 - \rho_{43} x_3- \cdots - \rho_{4n} x_n \right) \\
   \vdots  & = \vdots 
\end{split}
\end{equation}

Now note that distances from $W^s(O_1)$ and $W^u(O_1)$ are not effected by the coordinate change \eqref{coordchange}. If we 
now denote $\xi = (X_1, x_4, \ldots, x_n)$ then $|\xi|$ is the distance to the unstable manifold and $|\eta|$ is the distance
to the stable manifold (in the sup norm).

It now follows immediately from \eqref{neweqns} that inside the $\delta$ neighborhood of $O_1$ 
the unstable directions satisfy the estimates: 
$$
   \dot{x}_2 \le x_2 (\sigma_2 - \rho_{21}\sigma_1 - e^u)  \quad \text{ and } \quad  \dot{x}_3 \le x_3 (\sigma_3 - \rho_{31}\sigma_1 - e^u),
$$
where we can take $e^u >0$ arbitrarily small by starting with $\delta$ and $\epsilon$ small. 
Thus, by a simple application of Gronwall's inequality, solutions starting on $S_0$ and remaining 
in our $\delta$ neighborhood of $O_1$ must satisfy:
$$
     | \eta(t) | \le | \eta(0) | e^{(\lambda^{uu} -e^u) t},
$$
where $\lambda^{uu}$ is the maximal unstable eigenvalue, i.e.,
$$
    \lambda^{uu} = \max_{j=2,3} \{  \sigma_j - \rho_{j1}\sigma_1 \}.
$$
Thus if $T$ is defined by $|\eta (T)| = \delta$, then
$$
     T \ge\frac{1}{\lambda^{uu} - e^u} \ln(\frac{\delta}{\eta(0)}).
$$

Similarly, using Gronwall's inequality once again we obtain:
$$
   | \xi (t) | \le | \xi(0)| e^{(\lambda^{ls} + e^s) t },
$$
where $\lambda^{ls}$ is the ``leading" stable eigenvalue, i.e. since
the eigenvalues are real, the negative eigenvalue with the smallest absolute value.
In terms of our parameters:
$$
   \lambda^{ls} = \max \{ -\sigma_1,  \{ \sigma_j - \rho_{j1}\sigma_1, j=4,\ldots, n \}  \}.
$$
Thus
$$
     | \xi(T) | \le C(\delta) | \eta(0) |^{(\nu -e)}
 $$
 where $e$ can be taken arbitrarily small.
 By  the assumption that $O_1$ is dissipative, $\nu > 1$ and so we can make $\nu-e>1$.
 
\end{proof}


\subsection{Stable Heteroclinic Surface}
\label{stablesurface}
We are now in a position to prove that the union of the unstable 
manifolds of our system, restricted to the positive orthant, form an asymptotically stable forward invariant set under appropriate parameter restrictions.

\begin{thm*}[Theorem~\ref{maintheorem}, restated]
Suppose that inequalities \eqref{ew_unstable}, \eqref{ew_stable} and \eqref{dis} hold for each $k$, $1 \le k \le p$ and that each unstable manifold $\widetilde{W}^u(O_k)$ is contained 
in a compact forward invariant set as in Section~\ref{section:positive_invarian}.
Then $\Gamma \equiv \bigcup_{k=1}^p (\widetilde{W}_k^u \cup O_k)$ is asymptotically stable.
\end{thm*}

We consider the role each inequality of the hypothesis plays in the theorem.  The inequalities \eqref{ew_unstable} and \eqref{ew_stable} state that 
\begin{gather*}
0  <  \min_{i=1,2} \{ \sigma_{k+i}-\rho_{k+i,k}\sigma_k \} \, (\text{indices mod } p), \textrm{ for each }1 \le k \le p, \textrm{ and }\\
\lambda_j =  \sigma_j-\rho_{jk}\sigma_k<0, \quad \text{ for } j \ne k, k+1, k+2 \mod p,\textrm{ for each } 1 \le k \le p.
\end{gather*}

These inequalities are fundamental to the problem we are considering.  They ensure that there is a heteroclinic channel from each equilibrium $O_k$ to the equilibria $O_{k+1}$ and $O_{k+2}$ (the first inequality), and that there are no heteroclinic channels to the other equilibria (the second inequality).

The inequality \eqref{dis} states that for each $k$,\begin{equation*}
  \max_{i=1,2} \{ \sigma_{k+i}-\rho_{k+i,k}\sigma_k  \} 
        < \min_{j \neq k, k+1, k+2}  \{ |\sigma_j - \rho_{jk}\sigma_k|, \sigma_k \} \quad \text{ (indices mod $p$)}.
\end{equation*}  This is the dissipativity condition.  Informally, it may be taken to say that when a trajectory that is close to one of the unstable manifolds whose union is $\Gamma$ passes near one of the saddle fixed points, the distance of the trajectory to $\Gamma$ contracts exponentially.

It will become clear, from the proof of Theorem~\ref{maintheorem}, that \eqref{dis} is a much stronger condition than is necessary.  It ensures not only asymptotic stability, but a sort of monotonic asymptotic stability, such that whenever the trajectory passes near some $O_k$, its distance to $\Gamma$ contracts exponentially.  If, on the other hand, \eqref{dis} held for some, but not all, values of $k$, then the trajectory would at times contract exponentially towards $\Gamma$, and at other times drift away from $\Gamma$, and stability would depend on how these attractive and repulsive forces average over time.  Attempting to formulate a replacement condition for \eqref{dis} that is necessary as well as sufficient is extremely nontrivial.

The final hypothesis, that each unstable manifold is contained in a compact forward invariant set, is necessary.  We have seen one set of inequalities that ensure that such sets exist, \begin{equation*}
\frac{\sigma_{k+1}}{\rho_{k+1,k}} \le \frac{\sigma_{k+2}}{\rho_{k+2,k}}  \quad \text{ (indices mod $p$)},
\end{equation*} which are sufficient but not necessary.  In Section~\ref{parsers}, we see that there is a nonempty open region of parameter space where the hypotheses of Theorem~\ref{maintheorem} hold.  From our current discussion, we see that the theorem applies to a larger region.

We recall that $\Gamma$ is asymptotically stable if given any neighborhood $U$ of  $\Gamma$ 
(restricted to $\overline{\mathbb{R}_+^n})$ 
there exists an $\epsilon$-neighborhood of $\Gamma$, say $V_\epsilon (\Gamma) \subset 
\overline{\mathbb{R}_+^n}$, such that if $x_0 \in V_\epsilon (\Gamma)$ then 
$x(t,x_0) \in U$ for $t >0$ and 
$\lim_{t \rightarrow \infty} \text{dist}(x(t,x_0),\Gamma)=0$, 
where $x(t,x_0)$ is the solution of the initial value problem \eqref{eq:vectorfield} with $x(0,x_0)=x_0$.  When we speak of an open 
$\epsilon$-neighborhood of $\Gamma$, we are speaking of a set that is open in the subspace 
topology.  That is, an $\epsilon$-neighborhood in $\overline{\mathbb{R}_+^n}$ is an 
$\epsilon$-neighborhood in $\mathbb{R}^n$ intersected with $\overline{\mathbb{R}_+^n}$.

\begin{proof}[Proof of Theorem~\ref{maintheorem}]
For each $k=1,...,p$, let $V(O_k)$ be a sufficiently small $\delta$ neighborhood of $O_k$, such 
that Theorem~\ref{ConvergenceToManifold} can be applied within each $V(O_k)$ and $\delta$ 
does not depend on $k$.

Let $z_0$ be a representative point at an initial condition $\epsilon$-close to $\Gamma$.  
Choose $\epsilon$ such that $\epsilon<\delta$.  Then we can classify the dynamics as either 
local if $z_0 \in V(O_k)$ for some $k$, and global otherwise.

We observe that if $z_0=O_k$ for any $k$, its behavior is trivial, and likewise, if $z_0$ lies 
on a coordinate plane, it remains on that plane while converging exponentially to some $O_k$.  
We therefore assume without loss of generality that neither of these cases hold.

Suppose that $z_0 \notin V(O_k)$ for any $k$.  The point is $\epsilon$-close to $\Gamma$, 
and since $\Gamma$ is a finite union, we can say that $z_0$ is $\epsilon$-close to 
$\widetilde{W}_\alpha^u$, where $\alpha$ is fixed and depends on $z_0$.  Consider the projection 
of the system onto the three-dimensional subspace spanned by the axes $x_\alpha$, $x_{\alpha+1}$, 
and $x_{\alpha+2}$.  In three dimensions, the specific route a solution takes has not been important; 
a trajectory in the invariant set may go straight to a neighborhood of $O_{\alpha+2}$, 
or it may detour to $O_{\alpha+1}$, but the net result is the same (Theorem~\ref{3D}).  
We now formally differentiate between these two cases.

We consider two cases: either the positive semitrajectory of $z_0$ intersects $V(O_{\alpha+2})$ without first intersecting $V(O_{\alpha + 1}$) (case (i)), or the positive semitrajectory intersects $V(O_{\alpha+1})$, then intersects $V(O_{\alpha+2})$ (case (ii)).

Before proceeding, we recall the definitions of $S_0$ and $S_1$ given in Theorem~\ref{ConvergenceToManifold}, 
and similarly define such sections $S_0^q, S_1^q$ for $1 \le q \le p$.  Without loss of generality, we 
assume that $z_0 \in S_1^1$.

Suppose that case (i) occurs.  
For each $k=1,...,p$, let $\hat{V}(O_k)$ be the projection of $V(O_k)$ into the three dimensions spanned by 
$x_k, x_{k+1}, \text{ and }x_{k+2}$; note that $x_k$ is negligable for $k \ne \alpha, \alpha+1, \text{ and }\alpha+2$.  
Then the projection of the orbit onto $\mathbb{R}^3$ intersects $\hat{V}(O_{\alpha+2})$ before it can intersect 
$\hat{V}(O_{\alpha+1})$.  In $\mathbb{R}^3$, we know that all trajectories inside of $\widetilde{W}_{\alpha}^u$ 
that do not intersect $\hat{V}(P_{\alpha+1})$ come to a neighborhood of $O_{\alpha+2}$ in bounded time, where the bound 
does not depend on the initial condition.  We may consider the non-projected, full-dimensional space as 
a ``perturbation" of the projected space, and cite smooth dependence of initial conditions; the trajectory going through a slightly perturbed initial point corresponding to such a case must enter $V(O_{\alpha+2})$ in a 
well-behaved way.  In particular if $z_0$ belongs to $S_0 ^\alpha$ then a mapping from a neighborhood of $z_0$ on $S_0 ^\alpha$ to $S_0^{\alpha +2}$ is well defined and Lipschitz-continuous.

Suppose that case (ii) occurs.
Then once an orbit of $z_0$ enters $ V(O_{\alpha+1})$, it starts to manifest the dissipative behavior.  
In particular, if $d(x, W_{\alpha}^u)<\epsilon$, then after passing through 
$V(O_{\alpha+1})$, $d(\widetilde{x}, W_{\alpha+1}^u)<C\epsilon^\nu$, by Theorem~\ref{ConvergenceToManifold}.  
Once the representative point leaves  $V(O_{\alpha+1})$, we may apply (i), viewing its position after leaving 
the neighborhood as an initial condition that does not re-enter the neighborhood $V(O_{\alpha+1})$.  
Thus when the representative point finally enters $V(O_{\alpha+2})$, its distance to $\Gamma$ has been 
contracted by an order of $C\epsilon^{\nu_{\alpha+1} - e_{\alpha+1}}$, where $1<\nu_{\alpha+1} - e_{\alpha+1}$ and C absorbs both the constant $C$ from 
Theorem~\ref{ConvergenceToManifold} and a Lipschitz constant.

Suppose now that $z_0 \in V(O_k)$, where $k$ is now fixed.  Then the trajectory leaves $V(O_k)$ without 
increasing its distance from the unstable manifold, and passes into $V(O_{k+1})$ as just described.  
We may then apply Theorem~\ref{ConvergenceToManifold}.  As the trajectory passes through $V(O_k)$, 
its distance from the unstable manifold is contracted on an order of $\epsilon^{\nu_k - e_k}$.

Since the mapping contracts in the global dynamics and is Lipschitz (or contracting) in the local dynamics, 
simple inductions yields that as a representative point moves through the system, its distance from the 
manifold changes from $\epsilon$ to $c_1\epsilon^{\nu_1 - e_1}$ to $c_2\epsilon^{(\nu_1 - e_1)(\nu_2 - e_2)}$, and so on.  

For a fixed $i$, the value $\nu_i$, representing a ratio of eigenvalues, is likewise fixed.  The value $e_i$ is not; it depends on the distance between the representative point and the stable manifold as the trajectory enters $V(O_i)$, which changes from one instance to the next.  For a given $i$, however, there is some maximal value that $e_i$ can take, since the system is constantly contracting towards the manifold and $e_i$ goes to $0$ along with that distance.  Thus there exists a global value, $1<\nu < \nu_i - e_i$ for all $i$ and all $e_i$, such that passing from the first to the $p$-th unstable manifold is a contraction 
of order $c\epsilon^{\nu^p}$.
\end{proof}


\subsection{Existence of Parameter Sets}\label{parsers}

Throughout the paper, we have put a number of restrictions on the parameters of the system.  
One must ask whether the specified inequalities may be satisfied.
\begin{lem}
\label{lem:parameter_set}
There are sets of positive parameters values $\{ \sigma_i \}$, $i=1, \ldots, n$, and $\{ \rho_{jk} \}$, $j,k=1, \ldots, n$,
with non-empty interior such that the inequalities \eqref{ew_unstable}, 
\eqref{ew_stable}, \eqref{dis} and  \eqref{P23} hold.
\end{lem}
\begin{proof}
First note that given $\{ \sigma_i \}$ the inequalities \eqref{ew_unstable} and 
\eqref{ew_stable} are completely uncoupled and all trivially have positive solutions $\{ \rho_{jk} \}$.
For each $k$ and $i=1,2$ these are:
\begin{equation}\label{rhojk}
  \rho_{jk}  < \frac{ \sigma_j}{\sigma_k}
\end{equation} 
We note that with these inequalities the stable eigenvalues may be freely chosen to take any negative
value and the unstable eigenvalues any positive values. Since the restrictions 
\eqref{dis}  concern only relative orderings of those eigenvalues at each $O_k$ ($k$ fixed), it is clear
that \eqref{dis} is satisfied for open subsets of the previously chosen sets.

The final remaining inequality \eqref{P23} concerns on $\rho_{k+1,k}$ and $\rho_{k+2,k}$ and so
to finish the proof we need only to consider whether this restriction on those values is consistent with
the previous restrictions on those parameters, namely \eqref{ew_unstable} and \eqref{dis}, but not
\eqref{ew_stable}. Note that the inequalities in \eqref{ew_unstable} can be written as:
$$
    \frac{ \sigma_{k+1} }{ \rho_{k+1,k } } > \sigma_k  \quad 
    \text{ and }  \quad \frac{ \sigma_{k+2} }{ \rho_{k+2,k }} > \sigma_k.
$$
The constraint \eqref{P23} only requires that 
$$  
\frac{ \sigma_{k+1} }{ \rho_{k+1,k } }  \le \frac{ \sigma_{k+2} }{ \rho_{k+2,k} }.
$$
There is clearly no inconsistency in these inequalities.
The final inequalities \eqref{dis} involve each $\rho_{k+1,k}$ and $\rho_{k+2,k}$
independently of the others. First they require that for each $k$ and $i=1,2$ 
$$
    \sigma_k - \rho_{k+1,k} \sigma_k   < \sigma_k.
$$
This is equivalent to 
$$
   \frac{ \sigma_{k+i}}{\sigma_k} -1 < \rho_{k+1,k}
 $$
 which can clearly be satisfied along with \eqref{rhojk}.
If $\{ \rho_{k+i,k} \}$, $i=1,2$ have already been chosen to satisfy \eqref{ew_unstable}
and \eqref{P23}, then values for $\rho_{jk}$ can be chosen to satisfy \eqref{dis} by
simply choosing them sufficiently large, i.e., for each $j,k$, $j \neq k,k+1,k+2$, $\rho_{jk}$
must satisfy: 
$$
     \sigma_{k+1} - \rho_{k+i,k} \sigma_k < \rho_{jk} \sigma_k - \sigma_j,   \quad i=1,2.
$$
This can be rewritten as:
$$
\rho_{jk} >   \frac{\sigma_{k+i} + \sigma_j}{\sigma_k} - \rho_{k+i,k} ,  \quad i=1,2.
$$
Thus $\rho_{jk}$ can take any value greater than the maximum of these two values.
Previously we had only required (in \eqref{ew_stable}) that these parameters satisfy:
$$
    \rho_{jk} > \frac{\sigma_j}{\sigma_k}.
$$
Thus, with all other choices of consistent parameters, any large enough $\rho_{jk}$, will 
also be consistent.
\end{proof}

For example if $\sigma_i = \sigma$ are all the same, 
then \eqref{ew_unstable}, \eqref{ew_stable} and \eqref{P23} are satisfied if simple :
$$
0 < \rho_{k+2,k}  \le  \rho_{k+1,k} <1
$$
and
$$
\rho_{jk}  > 1,   \quad \text{ for } j \neq k, k+1,k+2.
$$
The dissipative requirement \eqref{dis} will be satisfied if further:
$$
2 \sigma - \rho_{k+i,k} < \rho_{jk} ,  \quad \text{ for } i=1,2 \text{ and } j \neq k,k+1,k+2.
$$

For instance, the parameter values: $\sigma_i = 1$, $\rho_{k+1,k} = .9$,  $\rho_{k+2,k} = .8$
and $\rho_{jk} =  1.3$ for  $j \neq k,k+1,k+2$, (indices mod $p$)  strictly satisfy all
of the inequalities.


\section{Conclusion}

We proved in the paper that under some conditions the generalized Lotka-Volterra system admits a two-dimensional  
attractor that consists of saddles and the unstable manifolds joining them into a heteroclinic system. 
Thus, trajectories inside the attractor manifest 
completely regular features. However, behavior of wandering trajectories in the basin of the attractor could be treated as 
weakly chaotic. Indeed, one can introduce an oriented graph with vertices identified with the saddle equilibrium points $O_k$ 
and edges identified with heteroclinic trajectories joining $O_k$ and $O_{k+1}$ (belonging to the coordinate plane $P_{k,k+1}$) or 
$O_k$ and $O_{k+2}$ (belonging to the plane $P_{k,k+2}$). It is possible to show that for each finite path through this 
graph there exists an open set of initial points in the basin
 such that the trajectory going through any of these points follows the corresponding heteroclinic trajectories. 
 The number of paths grows exponentially with the length, so the number of pieces of trajectories with different 
 behavior (in fact all of them 
are ($\epsilon$, $T$) separated for some values of $\epsilon$ and $T$)  grows as $T \rightarrow \infty$, the metric 
complexity function grows with time. 
A similar effect has been observed in \cite{masterslave}, where it 
was called weak transient chaos. We intend to describe its properties in another publication.

In the context of the motivating application, functional 
sequential dynamics in neural networks, the two-dimensional heteroclinic 
attractor $\Gamma$ in the phase space of dynamical 
system \eqref{eq:vectorfield} may be thought of as a mathematical image 
of diverse sequential dynamics based on the parallel performance of not one but two different 
modalities. Many interesting applications of this may be found in cognitive 
science. For example, the learning and performing of sensory-motor human 
behaviors in many situations require the integration or binding of the 
sequential stimuli of one modality with the sequential stimuli of another. 
This seems to be the case in one of the most important cognitive functions: 
sequential working memory. In performance of a cross-modal working memory task, 
there may be two 
sequentially discrete neural processes (different chains of metastable 
states) that represent simultaneous neural activities corresponding to 
cross-modal transfer of information in the working memory (see e.g. \cite{ohara}).

ACKNOWLEDGEMENT. V.A. was partially supported by the grant RNF 14-41-00044 of the Russian Science Foundation during his stay at Nizhny Novgorod University,
and by a Glidden Professorship Award during his stay at Ohio University.  The authors thank M. I. Rabinovich for bringing to their attention the problems considered in this manuscript.  We thank the referees for their comments and insight.



\appendix
\section{The Topological Form of $\Gamma$}

Theorem~\ref{formofgamma} states that $\Gamma$ depends on the parity of $p$; 
this is because the number of connected components in the boundary of $\Gamma$ depends on the parity of $p$.

\begin{prop}
\label{connectedcomponents}
If $p$ is even, then the boundary $\partial \Gamma$ has two connected components. 
If $p$ is odd, then $\partial \Gamma$ has one component.
\end{prop}
\begin{proof}
If $p$ is even, then one connected component of the boundary will include the 
trajectories connecting the saddles $O_n$ where $n$ is even, and another will 
include the trajectories connecting the saddles $O_m$, where $m$ is odd.

If $p$ is odd, then as in the previous case, the saddles $O_n$ where $n$ is even 
are contained in the same component, and the saddles $O_n$ where $n$ is odd are 
likewise contained in a single component. Furthermore, $O_{p-1}$ and $O_1$ are 
contained in the same component, where $p-1$ is even, since $O_1 = O_{p+1}$. 
Thus all saddles $O_n$ are contained in the same component.
\end{proof}

Figure~\ref{parityfigure} will help visualize the situation.

\begin{figure}[H]
\centering
\includegraphics[width=50mm]{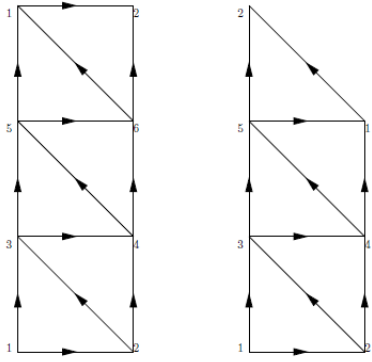}
\caption{A diagram representing $p=6$ (left) and $p=5$ (right). 
We clearly observe two distinct boundaries for $p=6$, and one for $p=5$.}
\label{parityfigure}
\end{figure}

It appears from the diagram that when $p$ is even, $\Gamma$ is a cylinder, and when $p$ is odd, 
$\Gamma$ is a M\"{o}bius strip. We will formalize that intuition. 
We first observe that the components whose union is the unstable manifold $\Gamma$ are not literally triangles, 
as they are depicted in Figure~\ref{parityfigure}, but rather curved surfaces.  
We cite a definition from algebraic topology (e.g. \cite{George}).

\begin{defn}
\label{triangulation}
Suppose $X$ is a compact Hausdorff space. A curved triangle in $X$ is a subspace $A$ of $X$ 
and a homeomorphism $H:T \rightarrow A$ where $T$ is a closed triangular region in the plane. 
A triangulation of $X$ is a collection of curved triangles $A_1,...,A_n$ in $X$ whose union is 
$X$ and such that for $i \ne j$, the intersection $A_i \cap A_j$ is either empty, or a vertex of 
both $A_i$ and $A_j$, or an edge of both. We also require, if $h_i$ is the homeomorphism 
associated with $A_i$, that when $A_i \cap A_j=e$ is an edge of both, then the map $h_j^{-1}h_i$ 
is a linear homeomorphism of the edge $h_i^{-1}(e)$ of $T_i$ with the edge $h_j^{-1}(e)$ of $T_j$.
\end{defn}

All compact surfaces have a triangulation, but we will prove in particular that the decomposition 
of $\Gamma$ into equilibria and the closure of their unstable manifolds, each restricted to the first orthant, forms a 
triangulation of $\Gamma$.

\begin{defn}
\label{heteroclinictriangles}
Consider the system of differential equations  \eqref{eq:vectorfield}. For each $\alpha$, we call the 
closure of $\widetilde{W}_\alpha^u \cup O_\alpha$ a heteroclinic triangle, $T_\alpha$.
\end{defn}

Note that:
$$
  T_\alpha = \widetilde{W}_\alpha^u \cup O_\alpha \cup O_{\alpha+1} \cup O_{\alpha+2} \cup \Gamma_{\alpha+1, \alpha+2}
$$
where $\Gamma_{\alpha+1, \alpha+2}$ is the heteroclinic orbit from $O_{\alpha+1}$ to $O_{\alpha+2}$.
This follows from the invariance of the $(\alpha, \alpha+1, \alpha+2)$-plane and Theorem~\ref{3D}.

\begin{thm}
\label{heteroclinictriangleshomeomorphism}
A heteroclinic triangle $T_\alpha$ is homeomorphic to a closed triangle in the plane.
\end{thm}

Certainly the boundary of a heteroclinic triangle $T_\alpha$, that is $O_\alpha$, $O_{\alpha+1}$, 
and $O_{\alpha+2}$, and the smooth paths connecting them, is homeomorphic to the boundary of 
a closed triangle in the plane. Further, the interior of $T_\alpha$ is $\widetilde{W}^u_\alpha$, by the unstable
manifold theorem and Theorem~\ref{3D}, homeomorphic to an open disk in the plane (and so also to the interior
of a triangular region).  The only complication may arise from the 
behavior of the unstable manifold near the edges or vertices. For instance, the $\widetilde{W}^u_\alpha$ could be ``folded"
as it approaches the edge  connecting $O_{\alpha+1}$ and $O_{\alpha+2}$ so that a neighborhood
of a point on the edge is not locally homeomorphic to a point on an edge of a closed triangle in the plane.

Consider the dynamics projected onto three consecutive dimensions; for simplicity of notation, 
we will use the standard $x, y, z$ coordinates, and label the 
saddle points on each axis $O_x$, $O_y$, 
and $O_z$, where there are heteroclinic connections 
$O_x \rightarrow O_y \rightarrow O_z$ and $O_x \rightarrow O_z$, which we denote as 
$\Gamma_{xy}$, $\Gamma_{xz}$, and $\Gamma_{yz}$.  We denote the corresponding heteroclinic triangle by $T_x$.

\begin {prop}
Each orbit in $\widetilde{W}^u(O_x)$ is uniquely identified with an angle $0 \le \phi \le \pi$.
\end{prop}

\begin{proof}
We define the usual $\delta$-neighborhoods $N_\delta(O_x)$, $N_\delta(O_y)$, 
and $N_\delta(O_z)$.  By the Unstable Manifold Theorem, we may choose $\delta$ small enough
that $W^u(O_x)$ is the graph of a function of $(y,z)$ in $N_\delta(O_k)$.  
Using the monotonicity of the $z$ coordinate inside the invariant region, we 
may also choose $\delta$ small enough such that once a 
trajectory leaves $N_\delta(O_x)$ it cannot return to it. 
By the previous observations, the other coordinates (other than $x$, $y$, and $z$) 
of $\widetilde{W}^u(O_x)$ are all zero.

Now consider that each orbit in $\widetilde{W}^u(O_x)$ has a unique intersection point $w$ with the 
boundary of $N_\delta(O_x)$. Since $\widetilde{W}^u(O_x)$ is a graph over $(y,z)$ consider
the projection $\bar{w}$ of $w$ onto $(y,z)$-plane. Let $\phi$ denote the angle
of the ray through the origin and $\bar{w}$ with the $z$-axis. Note that $\phi$ lies 
between $0$ and $\pi / 2$. The extreme angles $\phi =0$ and $\phi= \pi/2$ correspond
to the heteroclinic orbits $\Gamma_{xz}$ and $\Gamma_{xy}$ respectively. 
\end{proof}


\begin{defn}
For a point $p \in \widetilde{W}^u(O_x) \cup O_x$ define $d(p)$ to be the distance from $O_x$ to $p$ along
the orbit containing $p$. For a point $p \in \widetilde{W}^s(O_z) \cup O_z$ define $e(p)$ to be the 
distance from $p$ to $O_z$ along the orbit containing $p$. 
\end{defn}

\begin{prop}
The arc lengths $d$ and $e$ are finite and continuous where defined. 
\end{prop}
\begin{proof}
By basic existence theory the orbits are smooth and therefore arc length is locally well-defined on them.

Denote the solution with initial value $p$ as $p(t)$.
Recall from the Unstable Manifold Theorem that $p(t) \rightarrow O_x$ as $t \rightarrow - \infty$ 
and in fact 
\begin{equation}\label{Wu}
   \left| O_x - p(t) \right| \le C e^{(\lambda^{lu} - e^u) t}, \quad \text{ for } t \le 0,
\end{equation}
where $\lambda^{lu}$ is the leading unstable eigenvalue, i.e. in this case the minimum of 
$$
    \sigma_2 - \rho_{21} \sigma_1   \quad \text{ and } \quad \sigma_3 - \rho_{31} \sigma_1,
$$
and $e^u$ may be chosen to be arbitrarily small.

Now consider the length of the orbit:
\begin{equation}
\begin{split}
  d(p)  &=  \int_{- \infty}^0  \sqrt{ \dot{x}^2 + \dot{y}^2 + \dot{z}^2 }  \, dt  \\
      &=  \int_{ \tau}^0  \sqrt{ \dot{x}^2 + \dot{y}^2 + \dot{z}^2 }  \, dt  +  \int_{- \infty}^\tau  \sqrt{ \dot{x}^2 + \dot{y}^2 + \dot{z}^2 }  \, dt 
\end{split}
\end{equation}
If we substitute the equations \eqref{threedimensionalrestriction} into the integrals,
then substitute \eqref{Wu} into the second integral one easily sees that this integral goes
to zero as $\tau \rightarrow - \infty$. Thus $d(p)$ is finite.

Now consider the orbits starting at points $q$ close to $p$. Fix $\epsilon >0$ and chose
$\tau$ so that the remainder integral above is less than $\epsilon/3$. Since \eqref{Wu}
implies that $|q(t)-p(t)|$ goes to zero exponentially as $t \rightarrow -\infty$, we can make
the difference in the remainders less than $2 \epsilon/3$ if $q(\tau)$ is sufficiently close to
$p(\tau)$. This we can accomplish by requiring $q(0)$ and $p(0)$ sufficiently close (by continuous dependence on initial conditions). We can also make the integrals from $\tau$ to $0$ less than $\epsilon/3$ by choosing $q(0)$
close to $p(0)$. For such $q$, $|d(q) - d(p)| < \epsilon$.

The arc length $e(p)$  is finite for any $p$ not in $O_x \cup \Gamma_{xy} \cup O_y$ since
$O_z$ is a stable node in the $xyz$-subspace and all orbits  approach $O_z$ exponentially in time
and so the same type estimates as above hold here.
Continuity of this distance also follows by a similar proof as for $d$.
\end{proof}

\begin{cor}
All orbits on $T_x$  have finite length and this length is a continuous function of an initial point on $\widetilde{W}^u(O_x) \setminus (\Gamma_{xy} \cup O_y \cup \Gamma_{yz})$ 
\end{cor}
This follows from the continuity of $d$ and $e$ where they are defined.

\begin{lem}
The arc length $d$  can be extended continuously to the 
edge $\Gamma_{yz}$. The  arc length  is finite and uniformly bounded for all orbits that make up $T_x$. 
\end{lem}

\begin{proof}
Denote by $D_\phi$ the length of the orbit in $\widetilde{W}^u(O_x)$ identified by the angle $\phi$ for
$0 \le \phi < \pi/2$.

Note that distance is well defined along the edges  $\Gamma_{xy}$
and $\Gamma_{yz}$. In fact since $\Gamma_{xy} \subset \widetilde{W}^u(O_x) \cap \widetilde{W}^s(O_y)$,
and $\Gamma_{yz} \subset \widetilde{W}^u(O_y) \cap \widetilde{W}^s(O_z)$ the arguments above show
that the lengths of these two solutions arcs are finite.
Let $D_{xy}$ denote the length of $\Gamma_{xy}$ and $D_{yz}$ the length of $\Gamma_{yz}$.
Set 
$$
D_{\pi/2} = D_{xy} + D_{yz},
$$
i.e.\ the combined length of $\Gamma_{xy}$ and  $\Gamma_{yz}$.

For $p \in \Gamma_{yz}$ set:
$$
   d(p) = D_{\pi/2} - e(p),
$$
i.e. the length to $p$ along $\Gamma_{xy}$ and $\Gamma_{yz}$.
We claim that $d$ thus defined is continuous at $p \in \Gamma_{yz}$.

Fix $\epsilon >0$. 
Consider a $\delta$ (sup norm) neighborhood $N$ of $O_y$. For any
$\phi$ sufficiently close to $\pi/2$, the orbit on $T_x$ indexed by $\phi$
passes through $N$. Denote by $P^-$ the point where $\Gamma_{xy}$
intersects $N$ and by $P^+$ the point where $\Gamma_{yz}$ intersects $N$.
By continuity of $d$ and $e$ we may chose such a $\delta$ sufficiently small that:
$$
      D_{xy} - d(P^-) < \epsilon/6 \quad \text{ and } \quad  D_{yz} - e(P^+) < \epsilon/6.
$$ 
Further, require that $\delta$ be sufficiently small so that:
$\delta  \left(  \frac{2}{\lambda^{ls} + e^s} + \frac{1}{\lambda^u + e^u} \right) < \epsilon/6$.

Given $\phi$ sufficiently close to $\pi/2$, denote by $Q^-$ and $Q^+$ the
points where the orbit with angle $\phi$ intersects $N$. Again by continuity of 
$d$ and $e$ for all $\phi$ sufficiently close to $\pi/2$ we have:
$$
      |d(Q^-) - d(P^-)| < \epsilon/6 \quad \text{ and } \quad  e(Q^+) - e(P^+) < \epsilon/6.
$$

Let $q$ be a point sufficiently close to $p$ so that the above conditions on $\phi$ are
satisfied and such that $|e(q) - e(p)| < \epsilon/6$.

Now $d(q)$ will be equal $d(Q^-)$ plus the arc length from $Q^-$ to $Q^+$
plus the arc length from $Q^+$ to $q$. Denote these arc lengths by $\ell(Q^-,Q^+)$
and $\ell(Q^+,q)$ respectively.

Within the neighborhood $N$ we can transform coordinates to straighten the 
unstable manifold $\Gamma_{yz}$. In this subspace this unstable manifold is the 
graph of a function of $z$ only and has that has the form:
$$
y = \sigma_y + h(z).
$$
We thus straighten $\widetilde{W}^u(O_y)$ by the coordinate change $Y = y - \sigma_y + h(z)$.

By the Shilnikov variables technique \cite{Shil}, we have that the stable variables satisfy:
$$
   | (x(t),Y(t)) | <  \delta  e^{(\lambda^{ls} + e^s) t }, \quad 0 \le t \le \tau,
$$
while the unstable variable $z$ satisfies:
$$
     |z(t)| < \delta e^{(\lambda^u - e^u) (t- \tau ) }, \quad 0 \le t \le \tau,
$$
where $\tau$ is the time of passage through $N$.
We then obtain:
\begin{equation}
\begin{split}
  \ell(Q^-,Q^+)  &=  \int_0^\tau  \sqrt{ \dot{x}^2 + \dot{y}^2 + \dot{z}^2 }  \, dt  \\
      & <   \int_0^\tau   |\dot{x}| + |\dot{y}| + |\dot{z}|   \, dt  \\
       & < \delta    \int_0^\tau 2  e^{(\lambda^{ls} + e^s) t}  +  e^{(\lambda^u + e^u) (t- \tau ) } \\
       & < \delta    \frac{2}{\lambda^{ls} + e^s} + \frac{1}{\lambda^u + e^u} \\
       & < \frac{\epsilon}{6}.
      \end{split}
\end{equation}

We have that
\begin{equation}
\begin{split}
  d(q) - d(p) &= d(Q^-) + \ell(Q^-,Q^+) + \ell(Q^+,q) - \left( D_{xy} + D_{yz} - e(p) \right) \\
       &= d(Q^-) + \ell(Q^-,Q^+) + e(Q^+) - e(q)  \\
          & \hspace*{1cm} - D_{xy} + d(P^-) - d(P^-) - D_{yz} + e(P^+) - e(P^+) + e(p) \\
       &= d(Q^-) - d(P^-) + \ell(Q^-,Q^+) + e(Q^+) - e(P^+)  \\
         & \hspace*{1cm} + d(P^-)  - D_{xy} + e(P^+) - D_{yz} + e(p) - e(q) .
      \end{split}
\end{equation}
Combining the above estimates we have:  $|d(q) - d(p)| < \epsilon$. 

Now let $D_\phi$, $0 \le \phi \le \pi/2$ be the arc length of the orbit determined 
by the angle $\phi$. By the continuity of $d$ on all of $T_x$, except $O_x$ and the
continuity of $e$ in a neighborhood of $O_z$, $D_\phi$ depends continuously on $\phi$.
It is thus uniformly bounded.
\end{proof}


\begin{proof}[Proof of Theorem~\ref{heteroclinictriangleshomeomorphism}]

For each coordinate pair $p=(\phi (p),d(p))$, we normalize the arclength $d$ by leting $u = d / D_\phi$
Note that in this normalized distance $u(O_z) = 1$ and so 
$u$ is continuous on $T_x$.

We will show that the heteroclinic triangle $T_x$ with corners $O_x$, $O_y$, 
and $O_z$ is homeomorphic to the triangle $ABC$ in the plane
with corners at $A= (0,0)$, $B = (b,1/2)$ and $C = (1,0)$.

Now for a given point $w$ on $T$ identified by $(\phi,u)$, consider the map:
$$
H: w \mapsto (u,v) = (u,h(u,\phi)).
$$
where $v=h(u,\phi)$ is given by the continuous map: 
$$
   v =  \begin{cases} \frac{ u}{2b} \tan(\phi/2)  & \mbox{ if } 0 \le u \le b  \\
                               \frac{1-u}{2(1-b)} \tan(\phi/2)  & \mbox{ if } b < u \le 1.
          \end{cases}
$$

The map $H$ is a homeomorphism. It is one-to-one at all points. It maps
$O_x$, $O_y$ and $O_z$ onto $A$, $B$ and $C$ respectively.
At interior points it is a local homeomorphism because it is a composition
of local homeomorphisms. At the corners $O_x$ and $O_z$, considering
the mappings as polar coordinates make it clear that $H$ is a local homeomorphism
there. It is also clear that $H$ is one-to-one along the edges $\Gamma_{xz}$, $\Gamma_{xy}$
and $\Gamma_{yz}$. 
\end{proof}

Thus Figure~\ref{parityfigure} represents not merely an easy-to-understand representation of $\Gamma$, 
but a triangulation. We now investigate the orientibility of $\Gamma$. 
The most common definition of orientibility, in terms of normal vectors, is not useful in this situation, but having triangulated 
$\Gamma$, we may use instead a less common, but still standard definition (see e.g. \cite{orient}).

\begin{defn}
Consider some arbitrary triangle in the triangulation of a manifold. Assign to each triangle 
in the triangulation a value of clockwise.  This assigns corresponding directions to each 
side of each of the triangles.  Now let $\Delta_i$ and $\Delta_j$ be triangles sharing a side; 
observe that the common side has been assigned two {\bf different} directions, 
i.e. two adjacent triangles with the same orientation conflict on their shared side.  
If this can be done without contradiction, then the surface is orientable. 
If a contradiction is reached, the surface is non-orientable.
\end{defn}

It is a standard result of algebraic topology that orientibility is independent of the specific triangulation used.

Using this definition and the specific triangulation we have defined, the following lemma can be easily proven.

\begin{lem}
\label{orientability}
If $p$ is odd, $\Gamma$ is non-orientable. If $p$ is even, $\Gamma$ is orientable.
\end{lem}
\begin{proof}
Let $p$ be odd. We consider the triangulation by heteroclinic triangles. 
Start with $\Delta_1$, the triangle defined by $O_1$, $O_2$, and $O_3$ and without 
loss of generality assign it the clockwise orientation. Thus the direction on the $O_1-O_2$ 
side of the triangle is given the direction $O_2 \rightarrow O_1$.  Likewise, the triangle 
$\Delta_i$ defined by $O_p$, $O_1$, and $O_2$ is assigned the clockwise orientation, 
and the $O_1-O_2$ side of the triangle is given the direction $O_2 \rightarrow O_1$. 
But $\Delta_i$ and $\Delta_i$ are adjacent triangles, and the fact that they give the same direction to the $O_1-O_2$ 
side means that the surface is non-orientable.

Let $p$ be even. We consider the triangulation by heteroclinic triangles. Giving the triangle defined by 
$O_1$, $O_2$, and $O_3$ the clockwise orientation, 
and thus giving the $O_1-O_2$ edge the $O_2\rightarrow O_1$ direction, 
we clearly observe that making the $O_p$, $O_1$, $O_2$ triangle clockwise gives 
the $O_1-O_2$ boundary, the only place where orientibility could be broken, 
the conflicting $O_1 \rightarrow O_2$ direction.
\end{proof}

We cite one more result, a classification theorem \cite{George}.

\begin{thm}
\label{classification}
Given a compact connected triangulable $2$-manifold $Y$ with boundary such that 
$\partial Y$ has $k$ components, $Y$ is homeomorphic to $X$-with-$k$-holes, where 
$X$ is $S^2$ or the $n$-fold torus $T_n$ or the $m$-fold projective plane $P_m$.
\end{thm}

A ``hole" in the sense of the theorem is a set homeomorphic to an small $\epsilon$ open ball.

Our extensive build-up makes the proof of the theorem almost trivial.

\begin{thm}
If $p$ is odd, then $\Gamma$ is homeomorphic to a M\"{o}bius strip. 
If $p$ is even, $\Gamma$ is homeomorphic to a cylinder.
\end{thm}
\begin{proof}
Let $p$ be odd. Of the possible homeomorphic images named in Theorem~\ref{classification}, 
only the projective plane is non-orientable; since we know from Lemma~\ref{connectedcomponents}
 that $\Gamma$ has only one boundary component, $\Gamma$ is homeomorphic to the projective 
 plane with a ball removed, which is homeomorphic to the M\"{o}bius strip.

Let $p$ be even. Then $\Gamma$ is a rectangle with two of its edges identified. 
Since the ``right-hand" point of the bottom edge is identified with the 
``right-hand" point of the top edge, it is a cylinder.
\end{proof}

\end{document}